\documentclass{amsart}

\usepackage{amsmath,amssymb}
\usepackage{amsthm}
\usepackage{amsrefs}
\usepackage{indentfirst}
\usepackage{mathrsfs}
\usepackage{hyperref}
\usepackage{xcolor}
\usepackage[margin=1.4in]{geometry}
\usepackage{nicefrac}

\newcommand{\bR}{\mathbb{R}}
\newcommand{\bE}{\mathbb{E}}

\newcommand{\bNp}{\mathbb{N}_+}
\newcommand{\calB}{\mathcal{B}}

\newcommand{\calF}{\mathcal{F}}
\newcommand{\calT}{\mathcal{T}}
\newcommand{\norm}[1]{\left\Vert #1\right\Vert}
\newcommand{\gt}{\rightarrow}

\newcommand{\RR}{\mathbb{R}}

\newcommand{\Or}{\mathcal{O}}

\newcommand{\eps}{\epsilon}

\newlength{\leftstackrelawd}
\newlength{\leftstackrelbwd}
\def\leftstackrel#1#2{\settowidth{\leftstackrelawd}%
	{${{}^{#1}}$}\settowidth{\leftstackrelbwd}{$#2$}%
	\addtolength{\leftstackrelawd}{-\leftstackrelbwd}%
	\leavevmode\ifthenelse{\lengthtest{\leftstackrelawd>0pt}}%
	{\kern-.5\leftstackrelawd}{}\mathrel{\mathop{#2}\limits^{#1}}}

\numberwithin{equation}{section}
\newtheorem{theorem}{Theorem}[section]
\newtheorem{lemma}[theorem]{Lemma}
\newtheorem{remark}[theorem]{Remark}
\newtheorem{proposition}[theorem]{Proposition}
\newtheorem{assumption}[theorem]{Assumption}
\newtheorem{definition}[theorem]{Definition}

\allowdisplaybreaks[4]

\title[Spectral Barron Regularity for Static Schr\"odinger Equations on $\mathbb{R}^d$]{A Regularity Theory for Static Schr\"odinger Equations on $\mathbb{R}^d$ in Spectral Barron Spaces}
\author{Ziang Chen}
\address{(ZC) Department of Mathematics, Duke University, Box 90320, Durham, NC 27708.}
\email{ziang@math.duke.edu}
\author{Jianfeng Lu}
\address{(JL) Departments of Mathematics, Physics, and Chemistry, Duke University, Box 90320, Durham, NC 27708.}
\email{jianfeng@math.duke.edu}
\author{Yulong Lu}
\address{(YL)  Department of Mathematics and Statistics, Lederle Graduate Research Tower, University of Massachusetts, 710 N. Pleasant Street, Amherst, MA 01003.}
\email{lu@math.umass.edu}
\author{Shengxuan Zhou}
\address{(SZ)  Beijing International Center for Mathematical Research, Peking University, 5 Yiheyuan Road, Beijing, China, 100871.}
\email{zhoushx19@pku.edu.cn}
\date{\today}

\thanks{The work of ZC and JL is supported in part by the National Science Foundation via award DMS-2012286. YL thanks the support of National Science Foundation via award DMS-2107934. ZC and SZ thank Zexing Li for helpful discussions and communications.}

\begin{document}
	
	\begin{abstract}
		Spectral Barron spaces have received  considerable interest recently as it is the natural function space for approximation theory of two-layer neural networks with a dimension-free convergence rate. 
		In this paper we study the regularity of solutions to the whole-space static Schr\"odinger equation in spectral Barron spaces.  We prove that if the source of the equation lies in the spectral Barron space $\mathcal{B}^s(\RR^d)$ and the potential function admitting a non-negative lower bound decomposes as a positive constant plus a function in $\mathcal{B}^s(\RR^d)$, then the solution lies in the spectral Barron space $\mathcal{B}^{s+2}(\RR^d)$. 
	\end{abstract}

	\maketitle

	\section{Introduction}
	
	Numerical methods using neural networks for solving high-dimensional partial differential equations (PDEs) have achieved great success recently, see e.g., \cites{weinan2017deep, weinan2018deep, han2018solving, weinan2021algorithms}. A key advantage of these neural-network-based algorithms is that neural networks can approximate functions in certain classes efficiently, meaning that the complexity grows at most polynomially in the dimension. By contrast, conventional methods suffers from the curse of dimensionality (CoD). For example, the complexity for approximating a $d$-dimensional function using piecewise constant function with error tolerance $\eps$ is $\Or(\eps^{-d})$ that scales exponentially in $d$.  
	
	The efficiency of neural networks for approximating high-dimensional functions can even been observed for simple network structures, e.g., the two-layer neural networks,
	\begin{equation}\label{2layer-nn}
		g_n(x)=\sum_{j=1}^n a_j\sigma(w_j^\top x + b_j) + c,
	\end{equation}
	where $\sigma$ is the activation function and $n$ is the number of neurons. In the seminal work of Barron \cite{Barron93}, he shows that for a function $g$ satisfying
	\begin{equation}\label{C_g}
		C_g := \int_{\bR^d} |\hat{g}(\xi)|\cdot |\xi| d\xi <\infty,
	\end{equation}
	it can be approximated by a two-layer neural network $g_n$ of the form \eqref{2layer-nn} in $L^2$-norm, where complexity $n$ depends on the dimension $d$ at most polynomially if $C_g$ is viewed as a constant. 
	
	Along the direction of \cite{Barron93}, several specific types of function classes on which the neural network approximation in various norms does not suffer from CoD have been defined and illustrated in the literature, including particularly the spectral Barron space \cites{lu2021priori, siegel2020high, siegel2021optimal}, where the spectral Barron norms generalizing \eqref{C_g} are defined as some weighted $L^1$-norm of $\hat{u}$ or its discrete version, and the Barron space \cites{E19, Bach17} that generalizes \eqref{2layer-nn} into an integral representation with respect to an underlying probability measure on the parameter space and defines the Barron norm as path norm of the representation. With such function spaces specializing in high-dimensional problem, a natural question for the study of PDE is:
	
	\emph{If  the coefficients of a PDE lie in (spectral) Barron space, can the solution to the PDE also be guaranteed to be in (spectral) Barron space?}
	
	In this paper, we give a positive answer to the question above in the context of solving  the $d$-dimensional Schr\"odinger equation in the whole space:
	\begin{equation}\label{schrodinger_equ}
		-\Delta u + V u = f,\quad\text{in }\bR^d,
	\end{equation}
	Here  $V:\bR^d\gt\bR$ is the potential energy and $f:\bR^d\gt\bR$ is the source term. Morally speaking, our main result (Theorem~\ref{main_thm}) show that if $f$ is a spectral Barron function and $V$, with a non-negative lower bound on $\bR^d$, is the sum of a positive constant and a spectral Barron function, then the unique solution $u^*$ to \eqref{schrodinger_equ} is also in the spectral Barron space, with the order of spectral Barron regularity increased by $2$. We remark that a function in the spectral Barron spaces we consider has Fourier transform in $L^1(\bR^d)$, which implies that the function is bounded and decays to $0$ at infinity (by the Riemann-Lebesgue lemma). Therefore, Assumption \ref{assump_V}  on $V$  implies that $V$ is positive at  infinity, which guarantees the uniqueness of bounded solution and spectral Barron solution; see Proposition~\ref{injective_prop} for a precise statement.  An important consequence of Theorem~\ref{main_thm} is that there exists a two-layer neural network that approximates $u^\ast$ without curse of dimensionality; see Theorem~\ref{approx_thm}.
	
	\subsection{Related works} Regularity results of PDEs in Barron spaces have been studied in some recent works. In \cite{lu2021priori}, a solution theory for the Poisson equation and the Schr\"odinger equation on the bounded domain $\Omega=[0,1]^d$ with homogeneous Neumann boundary condition is establish in a type of spectral Barron space defined on $\Omega$ via cosine expansions. The same regularity result is later extended to the regularity estimate of ground state of the Schr\"odinger operator in \cite{lu2021priori2}. The work \cite{E20} proved regularity results for the screened Poisson equation $-\Delta u+\lambda^2 u=f$ and some time-dependent equations in the Barron space based on integral representation. 
	
	Besides the regularity estimates, another direction is to investigate the complexity of approximating PDE solutions using Barron functions or neural networks. It is obtained in \cite{chen2021representation} (representational) Barron complexity estimates for a general class of whole-space elliptic PDEs. A (deep) neural network complexity estimate for elliptic PDEs with homogeneous Dirichlet boundary condition is established in \cite{Marwah21}.
	
	\subsection{Notations} We use $\lvert x\rvert$ for the Euclidean norm of a vector $x\in \bR^d$ and use $B_r(x)=\{y\in \bR^d:\vert y-x\rvert<r\}$ for the open $\ell_2$ ball in $\bR^d$ centered at $x$ with radius $r$. For $i=1,2,\dots,d$, let $e_i\in\bR^d$ be the vector with the $i$-th entry being $1$ and other entries being $0$. For $g:\bR^d\gt\bR$, we denote by $\hat{g}$  its Fourier transform, given by 
	\begin{equation}\label{Fourier_trans}
		\hat{g}(\xi)=\frac{1}{(2\pi)^d}\int_{\bR^d} g(x)e^{-i x^\top \xi}dx.
	\end{equation}
	This is defined for $g\in L^1(\bR^d)$ and can be extended to tempered distributions. Note that  we  included a multiplicative constant $(2\pi)^{-d}$ in the definition \eqref{Fourier_trans} for the purpose of getting neater inverse Fourier transform:
	\begin{equation}\label{inv_Fourier_trans}
		g(x)=\int_{\bR^d} \hat{g}(x)e^{i x^\top\xi}d\xi.
	\end{equation}

	\section{Main results}
	
	Recall that regularity estimates for elliptic PDEs in Sobolev spaces are classical: Suppose that $V_{\min}\leq V(x)\leq V_{\max},\ \forall~x\in\bR^d$, for some $V_{\min},V_{\max}\in(0,+\infty)$, and that $f\in H^s(\bR^d)$ with $s\geq -1$, thanks to Lax-Milgram theorem, the Schr\"odinger equation \eqref{schrodinger_equ} admits a unique solution $u^*\in H^1(\bR^d)$. One can also obtain higher regularity for $u^*$ using  standard bootstrap argument. More specifically, if $u^*\in H^{s'}(\bR^d)$ and $V$ is sufficiently smooth, e.g., $V\in W^{s,\infty}(\bR^d)$, then it holds that
	\begin{equation*}
		(I-\Delta) u^* = f - u^*- V u^*\in H^{\min\{s',s\}}(\bR^d),
	\end{equation*}
	which implies that
	\begin{equation*}
		u^*=(I-\Delta)^{-1}(f - u^*- V u^*)\in H^{\min\{s',s\}+2}(\bR^d).
	\end{equation*}
	Therefore, one can conclude $u^*\in H^{s+2}(\bR^d)$ with
	\begin{equation*}
		\norm{u^*}_{H^{s+2}(\bR^d)}\leq C \norm{f}_{H^s(\bR^d)},
	\end{equation*}
	where $C$ is a constant depending on $V$, $d$, and $s$. Thus, $u^*$ has higher regularity than the source term $f$. Similar regularity results have also been studied for elliptic PDEs on bounded domains, see e.g., \cite{elliptic_PDE_fanghua}*{Theorem 5.27}.
	
	Our focus is to establish regularity results in the spectral Barron spaces for the Schr\"odinger equation \eqref{schrodinger_equ}. Let us first define the spectral Barron spaces as follows.
	
	\begin{definition}\label{def_spectral_Barron}
		Given $s\in\bR$, for a function $g:\bR^d\rightarrow\bR$, its spectral Barron norm is defined via
		\begin{equation*}
			\norm{g}_{\calB^s(\bR^d)}:=\int_{\bR^d} |\hat{g}(\xi)|\cdot(1+\lvert\xi\rvert^2)^{\frac{s}{2}} d\xi.
		\end{equation*}
		The spectral Barron space is the collection of  functions with finite spectral Barron norm:
		\begin{equation*}
			\calB^s(\bR^d)=\left\{g:\norm{g}_{\calB^s(\bR^d)}<\infty\right\}.
		\end{equation*}
	\end{definition}
	
	The spectral Barron space with index $s=1$ was first defined in the seminal work of Barron \cite{Barron93} and it has been further developed with general index $s$ in recent literature, see e.g., \cites{lu2021priori, siegel2020high, siegel2021optimal}. 
	Spectral Barron spaces are of particular interest for high-dimensional problems since a spectral Barron function can be efficiently approximated by a two-layer neural network without CoD, see e.g., \cite{Barron93} for approximation in $L^2$-norm and \cite{siegel2020high} for approximation in $H^k$-norm. A related but different notion of Barron space building upon integral representation was also proposed and studied in \cites{E19, Bach17}; see also \cites{parhi2021banach,parhi2021near} for a characterization of such space via Radon transform.
	
	Notice that by definition the spectral Barron space $\calB^s(\bR^d)$ is  a Banach space and the completeness follows from the fact that the spectral Barron norm is a weighted $L^1$-norm. A key difference between the spectral Barron space and the Sobolev space is that, $\norm{g}_{\calB^s(\bR^d)}$ is the $L^1$-norm of $\hat{g}(\xi)\cdot(1+\lvert\xi\rvert^2)^{s/2}$ while $\norm{g}_{H^s(\bR^d)}$ is the $L^2$-norm of $\hat{g}(\xi)\cdot(1+\lvert\xi\rvert^2)^{s/2}$. In particular, $\calB^s(\bR^d)$ is not a Hilbert space. The lack of Hilbert structure complicates the analysis of the existence and uniqueness of solutions in the spectral Barron space.

	To state our regularity theory for PDEs in spectral Barron spaces, we make the following assumption on  the potential $V$.
	
	\begin{assumption} \label{assump_V}
		Assume that the potential function $V$ satisfies the following:
		\begin{itemize}
			\item[(i)] $V(x)\geq 0,\ \forall~x\in\bR^d$;
			\item[(ii)] $V=\alpha+W$ where $\alpha>0$ is a constant and $W\in \calB^s(\bR^d)$. 
		\end{itemize}
	\end{assumption}

	Our main theorem can then be stated as follows:
	
	\begin{theorem}
		\label{main_thm}
		Suppose that Assumption~\ref{assump_V} holds with $s\geq 0$. Then for any $f\in\calB^s(\bR^d)$, there exists a unique solution $u^*$ in $\calB^{s+2}(\bR^d)$ and in addition it satisfies
		\begin{equation}\label{regularity_esti}
			\norm{u^*}_{\calB^{s+2}(\bR^d)}\leq C\norm{f}_{\calB^s(\bR^d)}.
		\end{equation}
		where $C$ is a constant depending on $V$, $d$, and $s$.
	\end{theorem}

	As a direct corollary of Theorem \ref{main_thm},  the solution to the Schr\"odinger equation \eqref{schrodinger_equ} can be approximated efficiently by a two-layer neural network on any bounded domain.
	
	\begin{theorem}\label{approx_thm}
		Under the same assumptions of Theorem~\ref{main_thm} and let $C$ being the constant in Theorem~\ref{main_thm}. Then for any $f\in\calB^s(\bR^d)$, any bounded domain $\Omega\subset \bR^d$, and any $n\in\bNp$, there is a cosine-activated two-layer neural network with $n$ hidden neurons,
		\begin{equation*}
			u_n(x)=\frac{1}{n}\sum_{j=1}^n a_j\cos(w_j^\top x + b_j),
		\end{equation*}
		satisfying
		\begin{equation}\label{approx_equ}
			\norm{u_n-u^*}_{H^1(\Omega)}\leq \frac{C  \sqrt{m(\Omega)}\norm{f}_{\calB^s(\bR^d)}}{n^{1/2}},
		\end{equation}
		where $u^*$ is the unique solution to \eqref{schrodinger_equ} in $\calB^{s+2}(\bR^d)$ and $m(\Omega)$ is the Lebesgue measure of $\Omega$.
	\end{theorem}
	
	\begin{remark}
		The approximation in Theorem~\ref{approx_thm} is stated in $H^1$-norm. However, stronger results could be expected. Since $u^*\in\calB^{s+2}(\bR^d)$, it holds that $u^*\in H^{s+2}(\Omega)$, and approximation in $H^{s+2}$-norm could hold. We refer to \cite{siegel2020high} for details. 
	\end{remark}

	\section{Proofs}
	This section is devoted to the proof of Theorem~\ref{main_thm}.  Due to the lack of Hilbert structure in the spectral Barron space $\mathcal{B}^s(\mathbb{R}^d)$, the standard Lax-Milgram theorem used to prove well-posedness of elliptic equations in Sobolev spaces can not be applied to the Barron spaces. Instead, we follow \cite{lu2021priori} and rewrite \eqref{schrodinger_equ} as an integral equation of the second kind 
	\begin{equation}\label{integral_equ}
		u+\calT_{\alpha,W}(u) = (\alpha-\Delta)^{-1} f,
	\end{equation}
	where
	\begin{equation}\label{operator_T}
		\calT_{\alpha,W}(u)=(\alpha-\Delta)^{-1}(W u).
	\end{equation}
	Our approach is to apply the Fredholm alternative to the integral equation \eqref{integral_equ}, thus the existence of solution follows from its uniqueness.  To this end, the essential step is to prove  the compactness of $\calT_{\alpha, W}$. The compactness was established in \cite{lu2021priori} for PDEs on bounded domains, but it becomes more challenging for unbounded domains, e.g. the whole space in our setting. In fact, when $W=1$, it is well-known that $\calT_{\alpha,1} = (\alpha-\Delta)^{-1}$ with $\alpha>0$ is not compact on $L^2(\bR^d)$ since it has continuous spectrum. On the contrary, we shall show that  $\calT_{\alpha,W} $ is indeed compact on $\calB^s(\bR^d)$ provided that $\hat{W}\in L^1(\bR^d)$ that is implies by Assumption~\ref{assump_V} (ii) with $s\geq 0$. 
	
	
	\subsection{Preliminary lemmas}
	
	We first present some preliminary lemmas for properties of the spectral Barron spaces and the operator $\calT_{\alpha,W}$.
	
	\begin{lemma}\label{prelim_lemma0}
		The following embeddings holds:
		\begin{itemize}
			\item[(i)] $\calB^0(\bR^d)\hookrightarrow L^\infty(\bR^d)$ with $\norm{g}_{L^\infty(\bR^d)}\leq\norm{g}_{\calB^0(\bR^d)}$.
			\item[(ii)] $\calB^{s'}(\bR^d)\hookrightarrow\calB^s(\bR^d)$ with $\norm{g}_{\calB^s(\bR^d)}\leq \norm{g}_{\calB^{s'}(\bR^d)}$ if $s\leq s'$.
		\end{itemize}
	\end{lemma}
	
	\begin{proof}
		They follow directly from Definition~\ref{def_spectral_Barron} and \eqref{inv_Fourier_trans}.
	\end{proof}
	
	\begin{lemma} \label{prelim_lemma1}
		If $\alpha>0$, then for any $g\in\calB^s(\bR^d)$, it holds that
		\begin{equation*}
			\norm{(\alpha-\Delta )^{-1} g}_{\calB^s(\bR^d)}\leq \frac{1}{\alpha}\norm{g}_{\calB^s(\bR^d)}.
		\end{equation*}
		and that
		\begin{equation*}
			\norm{(\alpha-\Delta )^{-1} g}_{\calB^{s+2}(\bR^d)}\leq \frac{1}{\min\{\alpha,1\}}\norm{g}_{\calB^s(\bR^d)}.
		\end{equation*}
	\end{lemma}
	
	\begin{proof}
		Denote $h=(\alpha-\Delta )^{-1} g$. Then $\hat{h}(\xi)=\frac{1}{\alpha +\lvert\xi\rvert^2} \hat{g}(\xi)$. One can hence compute that
		\begin{equation*}
			\norm{h}_{\calB^s(\bR^d)}=\int_{\bR^d} |\hat{h}(\xi)|\cdot(1+\lvert\xi\rvert^2)^{\frac{s}{2}}d\xi\leq \frac{1}{\alpha}\int_{\bR^d} |\hat{g}(\xi)|\cdot(1+\lvert\xi\rvert^2)^{\frac{s}{2}}d\xi=\frac{1}{\alpha}\norm{g}_{\calB^s(\bR^d)}.
		\end{equation*}
		and that
		\begin{align*}
			\norm{h}_{\calB^{s+2}(\bR^d)}&=\int_{\bR^d} |\hat{h}(\xi)|\cdot(1+\lvert\xi\rvert^2)^{\frac{s+2}{2}}d\xi=\int_{\bR^d} |\hat{g}(\xi)|\cdot\frac{1+\lvert\xi\rvert^2}{\alpha+\lvert\xi\rvert^2}\cdot(1+\lvert\xi\rvert^2)^{\frac{s}{2}}d\xi\\
			&\leq\frac{1}{\min\{\alpha,1\}}\int_{\bR^d} |\hat{g}(\xi)|\cdot(1+\lvert\xi\rvert^2)^{\frac{s}{2}}d\xi=\frac{1}{\min\{\alpha,1\}}\norm{g}_{\calB^s(\bR^d)}. 
		\end{align*}
	\end{proof}
	
	\begin{remark}
		Since $g\in \calB^s(\bR^d)$ is a real-valued function, $(\alpha-\Delta)^{-1}g$ with $\alpha>0$ must also be real-valued. This is because that $-\Delta u + \alpha u = 0$ only has trivial solution in the space of tempered distributions, which can be seen directly by taking Fourier transform $\hat{u}=\frac{0}{\alpha+\lvert\xi\rvert^2}=0$.
	\end{remark}
	
	\begin{lemma}\label{prelim_lemma2}
		Suppose that $W\in\calB^s(\bR^d)$ with $s\geq 0$. Then for any $u\in\calB^s(\bR^d)$, it holds that
		\begin{equation*}
			\norm{W u}_{\calB^s(\bR^d)}\leq 2^{\frac{s}{2}}\norm{W}_{\calB^s(\bR^d)}\norm{u}_{\calB^s(\bR^d)}.
		\end{equation*}
	\end{lemma}
	
	\begin{proof}
		It follows from
		\begin{equation*}
			\widehat{Wu}(\eta)=\hat{W}\ast\hat{u}(\eta)=\int_{\bR^d} \hat{W}(\xi)\hat{u}(\eta-\xi)d\xi,
		\end{equation*}
		that
		\begin{align*}
			\norm{W u}_{\calB^s}&\leq \int_{\bR^d\times\bR^d} |\hat{W}(\xi)|\cdot|\hat{u}(\eta-\xi)|\cdot (1+\lvert\eta\rvert^2)^{\frac{s}{2}} d\xi d\eta\\
			&\leq \int_{\bR^d\times\bR^d} |\hat{W}(\xi)|\cdot|\hat{u}(\eta-\xi)|\cdot (1+2\lvert\xi\rvert^2+2\lvert\eta-\xi\rvert^2)^{\frac{s}{2}} d\xi d\eta\\
			& \leq 2^{\frac{s}{2}}\int_{\bR^d\times\bR^d} |\hat{W}(\xi)|\cdot|\hat{u}(\eta-\xi)|\cdot (1+\lvert\xi\rvert^2)^{\frac{s}{2}}\cdot (1+\lvert\eta-\xi\rvert^2)^{\frac{s}{2}} d\xi d\eta\\
			& =2^{\frac{s}{2}}\int_{\bR^d} |\hat{W}(\xi)|\cdot(1+\lvert\xi\rvert^2)^{\frac{s}{2}} d\xi \int_{\bR^d} |\hat{u}(\xi)|\cdot(1+\lvert\xi\rvert^2)^{\frac{s}{2}} d\xi\\
			& = 2^{\frac{s}{2}}\norm{W}_{\calB^s}\norm{u}_{\calB^s}. 
		\end{align*}
	\end{proof}
	
	\medskip 
	
	\subsection{Compactness of $\calT_{\alpha,W}$}
	
	Lemma~\ref{prelim_lemma1} and Lemma~\ref{prelim_lemma2} imply that the operator $\calT_{\alpha,W}$ defined in \eqref{operator_T} is bounded in $\calB^s(\bR^d)$ if $\alpha>0$ and $W\in\calB^s(\bR^d)$ with $s\geq 0$. We now show that this operator is compact with a more careful analysis.
	
	\begin{proposition}\label{compact_prop}
		Suppose that $\alpha>0$ and $W\in\calB^s(\bR^d)$ with $s\geq 0$. Then the operator $\calT_{\alpha, W}:\calB^s(\bR^d)\gt\calB^s(\bR^d)$ defined in \eqref{operator_T} is compact.
	\end{proposition}
	
	To prove that $\calT_{\alpha, W}:\calB^s(\bR^d)\gt\calB^s(\bR^d)$ is compact, it suffices to show that the image of the closed unit ball in $\calB^s(\bR^d)$,
	\begin{equation}\label{TX}
		\left\{\calT_{\alpha, W} (u): \norm{u}_{\calB^s(\bR^d)}\leq 1\right\},
	\end{equation}
	is relatively compact in $\calB^s(\bR^d)$. Notice that $\calB^s(\bR^d)$ is complete, which implies that a subset of $\calB^s(\bR^d)$ is relatively compact if and only if it is totally bounded. Therefore, it suffices to prove the total boundedness of
	\begin{equation}\label{funct_class_F}
		\calF:=\left\{\widehat{\calT_{\alpha, W} (u)}(\xi)\cdot (1+\lvert\xi\rvert^2)^{\frac{s}{2}}: \norm{u}_{\calB^s(\bR^d)}\leq 1\right\}\subset L^1(\bR^d),
	\end{equation}
	where we translate the $\calB^s$-norm into the usual $L^1$-norm. The following Kolmogorov-Riesz theorem will be useful for establishing the total boundedness.
	
	\begin{theorem}[Kolmogorov-Riesz theorem {\cite{hanche2010kolmogorov}*{Theorem 5}}]\label{Kolmogorov_Riesz_thm}
		For $p\in [1,\infty)$, a subset $\calF \subset L^p(\bR^d)$ is totally bounded if and only if the following three conditions hold:
		\begin{itemize}
			\item[(i)] $\calF$ is bounded;
			\item[(ii)] For any $\epsilon>0$, there exists $R>0$ such that
			\begin{equation*}
				\int_{\lvert x\rvert>R} |f(x)|^p dx<\epsilon^p,\quad \forall\ f\in \calF;
			\end{equation*}
			\item[(iii)] For any $\epsilon>0$, there exists $\delta>0$, such that
			\begin{equation*}
				\int_{\bR^d} |f(x+y)-f(x)|^p dx<\epsilon^p,\quad \forall\ f\in\calF,\ \lvert y\rvert<\delta.
			\end{equation*}
		\end{itemize}
	\end{theorem}
	
	We then prove Proposition~\ref{compact_prop} using Theorem~\ref{Kolmogorov_Riesz_thm}.
	
	\begin{proof}[Proof of Proposition~\ref{compact_prop}]
		As discussed above, the compactness of $\calT_{\alpha, W}$ follows from the total boundedness of $\calF$ defined in \eqref{funct_class_F}. Therefore, it suffices to verify the three conditions in Theorem~\ref{Kolmogorov_Riesz_thm} for $\calF$. We verify them one by one. 

		
		$\bullet$ \emph{Verification of Condition (i) in Theorem~\ref{Kolmogorov_Riesz_thm}:} For any $\norm{u}_{\calB^s(\bR^d)}\leq 1$, since
		\begin{equation*}
			\widehat{\calT_{\alpha, W} (u)}(\xi)=\frac{1}{\alpha+ \lvert\xi\rvert^2}\widehat{Wu}(\xi),
		\end{equation*}
		we have that
		\begin{equation*}
			\begin{split}
				\norm{\widehat{\calT_{\alpha,W} (u)}(\xi)\cdot(1+\lvert\xi\rvert^2)^{\frac{s}{2}}}_{L^1(\bR^d)}&\leq \norm{(\alpha-\Delta)^{-1} Wu}_{\calB^s(\bR^d)}\\
				&\leq \frac{1}{\alpha}\norm{Wu}_{\calB^s(\bR^d)}\\
				&\leq\frac{2^{\frac{s}{2}}}{\alpha}\norm{W}_{\calB^s(\bR^d)}\norm{u}_{\calB^s(\bR^d)}\\
				&\leq\frac{2^{\frac{s}{2}}}{\alpha}\norm{W}_{\calB^s(\bR^d)},
			\end{split}
		\end{equation*}
		where we used Lemma~\ref{prelim_lemma1} and Lemma~\ref{prelim_lemma2}. Therefore, $\calF$ is bounded in $L^1(\bR^d)$.
		
		$\bullet$ \emph{Verification of Condition (ii) in Theorem~\ref{Kolmogorov_Riesz_thm}:} For any $\epsilon>0$, there exists $R>0$ such that $\frac{1}{\alpha+\lvert\xi\rvert^2}<\epsilon$ for any $\lvert\xi\rvert>R$. Then for any $\norm{u}_{\calB^s(\bR^d)}\leq 1$, it holds that
		\begin{equation*}
			\begin{split}
				\int_{\lvert\xi\rvert>R}|\widehat{\calT_{\alpha, W} (u)}(\xi)|\cdot(1+\lvert\xi\rvert^2)^{\frac{s}{2}}d\xi&\leq \int_{\lvert\xi\rvert>R}\frac{1}{\alpha+\lvert\xi\rvert^2}\cdot |\widehat{W u}(\xi)|\cdot (1+\lvert\xi\rvert^2)^{\frac{s}{2}}d\xi\\
				&\leq \epsilon\int_{\lvert\xi\rvert>R}|\widehat{W u}(\xi)|\cdot (1+|\xi|^2)^{\frac{s}{2}}d\xi\\
				&\leq \epsilon\norm{Wu}_{\calB^s(\bR^d)}\\
				&\leq\epsilon 2^{\frac{s}{2}}\norm{W}_{\calB^s(\bR^d)}\norm{u}_{\calB^s(\bR^d)}\\
				&\leq\epsilon2^{\frac{s}{2}}\norm{W}_{\calB^s(\bR^d)},
			\end{split}
		\end{equation*}
		where we also used Lemma~\ref{prelim_lemma2}.

		$\bullet$ \emph{Verification of Condition (iii) in Theorem~\ref{Kolmogorov_Riesz_thm}:} Since Condition (ii) in Theorem~\ref{Kolmogorov_Riesz_thm} holds, for any $\epsilon>0$, there exists $R>0$ such that
		\begin{equation}\label{esti1}
			\int_{\lvert\xi\rvert> R} |\widehat{\calT_{\alpha,W}(u)}(\xi)|\cdot (1+\lvert\xi\rvert^2)^{\frac{s}{2}} d\xi\leq \epsilon,\quad\forall\ \norm{u}_{\calB^s(\bR^d)}\leq 1.
		\end{equation}
		Set
		\begin{equation}\label{L1L2}
			L_1:=\max_{\lvert\xi\rvert\leq 2R}\frac{(1+\lvert\xi\rvert^2)^{\frac{s}{2}}}{\alpha+ \lvert\xi\rvert^2}\quad\text{and}\quad L_2:=\int_{\lvert\xi\rvert\leq 2R}d\xi.
		\end{equation}
		It follows from $W\in \calB^s(\bR^d)$ with $s\geq 0$ that $\hat{W}\in L^1(\bR^d)$. According to \cite{folland1999real}*{Proposition 8.17}, there exists $\varphi\in C_c^{\infty}(\bR^d)$ satisfying 
		\begin{equation}\label{esti9}
			\norm{\hat{W} - \varphi}_{L^1(\bR^d)}\leq \frac{\eps}{L_1}.
		\end{equation}
		Note that $\xi\mapsto\frac{(1+\lvert\xi\rvert^2)^{\frac{s}{2}}}{\alpha+\lvert\xi\rvert^2}$ is continuous, and is hence uniformly continuous on any compact subsets of $\bR^d$. One also has that $\varphi\in C_c^{\infty}(\bR^d)$ is uniformly continuous on $\bR^d$. Thus, there exists some $\delta<R$, such that
		\begin{equation}\label{esti3}
			\left|\frac{(1+\lvert\xi\rvert^2)^{\frac{s}{2}}}{\alpha+\lvert\xi\rvert^2}-\frac{(1+\lvert\xi'\rvert^2)^{s/2}}{\alpha+\lvert\xi'\rvert^2}\right|\leq \epsilon,\quad \forall~\lvert\xi\rvert\leq 3R,\ \lvert\xi'\rvert\leq 3R,\ \lvert\xi-\xi'\rvert<\delta,
		\end{equation}
		and that
		\begin{equation}\label{esti4}
			|\varphi(\xi)-\varphi(\xi')|<\frac{\epsilon}{L_1 L_2},\quad \forall~\xi,\xi'\in\bR^d,\ \lvert\xi-\xi'\rvert<\delta.
		\end{equation}
		Consider any $\lvert y\rvert<\delta<R$ and any $\norm{u}_{\calB^s(\bR^d)}\leq 1$, we have that
		\begin{equation}\label{esti8}
			\begin{split}
				&\int_{\bR^d}\left|\widehat{\calT_{\alpha,W} (u)}(\xi+y)\cdot (1+\lvert\xi+y\rvert^2)^{\frac{s}{2}}-\widehat{\calT_{\alpha,W} (u)}(\xi)\cdot (1+\lvert\xi\rvert^2)^{\frac{s}{2}}\right|d\xi\\
				\leq & \int_{\lvert\xi\rvert>2R}|\widehat{\calT_{\alpha,W}(u)}(\xi+y)|\cdot (1+\lvert\xi+y\rvert^2)^{\frac{s}{2}}d\xi\\
				&\qquad +\int_{\lvert\xi\rvert>2 R}|\widehat{\calT_{\alpha,W} (u)}(\xi)|\cdot (1+\lvert\xi\rvert^2)^{\frac{s}{2}}d\xi\\
				&\qquad +\int_{\lvert\xi\rvert\leq 2R}\left|\widehat{\calT_{\alpha,W} (u)}(\xi+y)\cdot (1+\lvert\xi+y\rvert^2)^{\frac{s}{2}}-\widehat{\calT_{\alpha,W} (u)}(\xi)\cdot (1+\lvert\xi\rvert^2)^{\frac{s}{2}}\right|d\xi\\
				\leq & 2\int_{\lvert\xi\rvert> R}|\widehat{\calT_{\alpha,W} (u)}(\xi)|\cdot (1+\lvert\xi\rvert^2)^{\frac{s}{2}}d\xi\\
				&\qquad+\int_{\lvert\xi\rvert\leq 2R}\left|\widehat{\calT_{\alpha,W} (u)}(\xi+y)\cdot (1+\lvert\xi+y\rvert^2)^{\frac{s}{2}}-\widehat{\calT_{\alpha,W} (u)}(\xi)\cdot (1+\lvert\xi\rvert^2)^{\frac{s}{2}}\right|d\xi\\
				\leq & 2\epsilon+\int_{\lvert\xi\rvert\leq 2R}\left|\widehat{\calT_{\alpha,W} (u)}(\xi+y)\cdot (1+\lvert\xi+y\rvert^2)^{\frac{s}{2}}-\widehat{\calT_{\alpha,W} (u)}(\xi)\cdot (1+\lvert\xi\rvert^2)^{\frac{s}{2}}\right|d\xi,
			\end{split}
		\end{equation}
		where the last inequality follows from \eqref{esti1}. Then we estimate the second term in the last line above. We have that
		\begin{equation}\label{esti5}
			\begin{split}
				&\left|\widehat{\calT_{\alpha,W} (u)}(\xi+y)\cdot (1+\lvert\xi+y\rvert^2)^{\frac{s}{2}}-\widehat{\calT_{\alpha,W} (u)}(\xi)\cdot (1+\lvert\xi\rvert^2)^{\frac{s}{2}}\right|\\
				=&\left|\widehat{Wu}(\xi+y)\cdot \frac{(1+\lvert\xi+y\rvert^2)^{\frac{s}{2}}}{\alpha+\lvert\xi+y\rvert^2}-\widehat{Wu}(\xi)\cdot\frac{(1+\lvert\xi\rvert^2)^{\frac{s}{2}}}{\alpha+\lvert\xi\rvert^2}\right|\\
				\leq & |\widehat{Wu}(\xi+y)|\cdot \left|\frac{(1+\lvert\xi+y\rvert^2)^{\frac{s}{2}}}{\alpha+\lvert\xi+y\rvert^2}-\frac{(1+\lvert\xi\rvert^2)^{\frac{s}{2}}}{\alpha+\lvert\xi\rvert^2}\right|+\frac{(1+\lvert\xi\rvert^2)^{\frac{s}{2}}}{\alpha+\lvert\xi\rvert^2}\cdot |\widehat{Wu}(\xi+y)-\widehat{Wu}(\xi)|.
			\end{split}
		\end{equation}
		According to \eqref{esti3} and Lemma~\ref{prelim_lemma2}, it holds that
		\begin{equation}\label{esti6}
			\begin{split}
				&\int_{\lvert\xi\rvert\leq 2R}|\widehat{Wu}(\xi+y)|\cdot \left|\frac{(1+\lvert\xi+y\rvert^2)^{\frac{s}{2}}}{\alpha+\lvert\xi+y\rvert^2}-\frac{(1+\lvert\xi\rvert^2)^{\frac{s}{2}}}{\alpha+\lvert\xi\rvert^2}\right|d\xi\\
				\leq & \epsilon \int_{\lvert\xi\rvert\leq 2R} |\widehat{Wu}(\xi+y)|d\xi \leq \epsilon \norm{W u}_{\calB^0(\bR^d)}\leq\epsilon \norm{W}_{\calB^0(\bR^d)}\norm{u}_{\calB^0(\bR^d)}\\
				\leq &\epsilon\norm{W}_{\calB^s(\bR^d)}\norm{u}_{\calB^s(\bR^d)}\leq \epsilon\norm{W}_{\calB^s(\bR^d)}.
			\end{split}
		\end{equation}
		By \eqref{L1L2}, \eqref{esti9}, and \eqref{esti4}, it holds that
		\begin{equation}\label{esti7}
			\begin{split}
				&\int_{\lvert\xi\rvert\leq 2R}\frac{(1+\lvert\xi\rvert^2)^{\frac{s}{2}}}{\alpha+\lvert\xi\rvert^2}\cdot |\widehat{Wu}(\xi+y)-\widehat{Wu}(\xi)|d\xi\\
				=&\int_{\lvert\xi\rvert\leq 2R}\frac{(1+\lvert\xi\rvert^2)^{\frac{s}{2}}}{\alpha+\lvert\xi\rvert^2}\cdot \left|\int_{\bR^d} \hat{u}(\eta)\hat{W}(\xi+y-\eta)d\eta-\int_{\bR^d} \hat{u}(\eta)\hat{W}(\xi-\eta)d\eta\right|d\xi\\
				\leq& L_1\int_{\lvert\xi\rvert\leq 2R}\int_{\bR^d} |\hat{u}(\eta)|\cdot|\hat{W}(\xi+y-\eta)-\hat{W}(\xi-\eta)|d\eta d\xi\\
				\leq& L_1\int_{\lvert\xi\rvert\leq 2R}\int_{\bR^d} |\hat{u}(\eta)|\cdot|\varphi(\xi+y-\eta)-\varphi(\xi-\eta)|d\eta d\xi\\
				&\qquad + L_1\int_{\lvert\xi\rvert\leq 2R}\int_{\bR^d} |\hat{u}(\eta)|\cdot|\hat{W}(\xi+y-\eta)-\varphi(\xi+y-\eta)|d\eta d\xi\\
				&\qquad + L_1\int_{\lvert\xi\rvert\leq 2R}\int_{\bR^d} |\hat{u}(\eta)|\cdot|\hat{W}(\xi-\eta)-\varphi(\xi-\eta)|d\eta d\xi\\
				\leq & \frac{\epsilon}{L_2} \int_{\lvert\xi\rvert\leq 2R}d\xi \int_{\bR^d} |\hat{u}(\eta)|d\eta + 2 L_1 \norm{\hat{W}-\varphi}_{L^1(\bR^d)}\int_{\bR^d} |\hat{u}(\eta)| d\eta \\
				\leq & 3\epsilon\int_{\bR^d} |\hat{u}(\eta)|d\eta \leq 3\epsilon\norm{u}_{\calB^s(\bR^d)}\leq 3\epsilon.
			\end{split}
		\end{equation}
		Combining \eqref{esti5}, \eqref{esti6}, and \eqref{esti7}, we obtain that
		\begin{equation*}
			\begin{split}
				&\int_{\lvert\xi\rvert\leq 2R}\left|\widehat{\calT_{\alpha,W} (u)}(\xi+y)\cdot (1+\lvert\xi+y\rvert^2)^{\frac{s}{2}}-\widehat{\calT_{\alpha,W} (u)}(\xi)\cdot (1+\lvert\xi\rvert^2)^{\frac{s}{2}}\right|d\xi\\
				&\qquad\qquad \leq (3+\norm{W}_{\calB^s(\bR^d)})\cdot\epsilon,
			\end{split}
		\end{equation*}
		which combined with \eqref{esti8} yields that
		\begin{equation*}
			\begin{split}
				&\int_{\bR^d}\left|\widehat{\calT_{\alpha,W} (u)}(\xi+y)\cdot (1+\lvert\xi+y\rvert^2)^{\frac{s}{2}}-\widehat{\calT_{\alpha,W} (u)}(\xi)\cdot (1+\lvert\xi\rvert^2)^{\frac{s}{2}}\right|d\xi\\
				&\qquad\qquad\leq (5+\norm{W}_{\calB^s}(\bR^d))\cdot\epsilon,
			\end{split}
		\end{equation*}
		for any $\lvert y\rvert<\delta$ and $\norm{u}_{\calB^s(\bR^d)}\leq 1$. This completes the proof.
	\end{proof}

	\subsection{Proof of the main results}
	
	We finish the proof of Theorem~\ref{main_thm} in this subsection. We first need to establish the existence of the solution to \eqref{integral_equ} that is equivalent to the original PDE \eqref{schrodinger_equ} in $\calB^s(\bR^d)$.
	
	\begin{proposition}\label{bdd_inverse_prop}
		Suppose that Assumption~\ref{assump_V} holds with $s\geq 0$. Then the operator 
		\begin{equation*}
			(I+\calT_{\alpha, W})^{-1}:\calB^s(\bR^d)\gt\calB^s(\bR^d)
		\end{equation*}
		is bounded.
	\end{proposition}
	
	Since $\calT_{\alpha, W}:\calB^s(\bR^d)\gt\calB^s(\bR^d)$ has been proved as compact in Proposition~\ref{compact_prop}, $I+\calT_{\alpha, W}$ is a Fredholm operator. Therefore, to show that $I+\calT_{\alpha, W}$ has bounded inverse, it suffices to show that $I+\calT_{\alpha, W}$ is injective, which is established in the following proposition.
	
	\begin{proposition}\label{injective_prop}
		Suppose that Assumption~\ref{assump_V} holds with $s\geq 0$. Then the operator 
		\begin{equation*}
			I+\calT_{\alpha, W}:\calB^s(\bR^d)\gt\calB^s(\bR^d)
		\end{equation*}
		is injective.
	\end{proposition}

	\begin{proof} 
		Suppose that there exists some $u\in\calB^s(\bR^d)$ such that
		\begin{equation*}
			u+\calT_{\alpha,W}(u)=0,
		\end{equation*}
		which is equivalent to
		\begin{equation*}
			-\Delta u + V u = 0,
		\end{equation*}
		where $V=\alpha+W$. Since $s\geq 0$, we have that $u\in L^\infty(\bR^d)$ by Lemma~\ref{prelim_lemma0}. Furthermore, $u$ and $V$ are both continuous as the Fourier transform of a function in $\calB^s(\bR^d)$ is in $L^1(\bR^d)$.
		
		Suppose that $u$ is not identically zero, which means that $u(x_0)\neq 0$ holds for some $x_0\in \bR^d$. It follows from $\hat{W}\in L^1(\bR^d)$ and the Riemann-Lebesgue lemma that $\lim_{\lvert x\rvert\rightarrow\infty}W(x)=0$, which implies that there exists some $R\geq\lvert x_0\rvert$ such that
		\begin{equation*}
			\inf_{\lvert x\rvert\geq R} V(x)\geq \frac{\alpha}{2}.
		\end{equation*}
		Note that Assumption~\ref{assump_V}~(i) states that $V(x)\geq 0,\ \forall~x\in\bR^d$. According to weak maximal principle, we have for $r>0$ that
		\begin{equation}\label{max_principal}
			\sup_{\lvert x\rvert\leq r}|u(x)|=\sup_{\lvert x\rvert= r}|u(x)|.
		\end{equation}
		By \eqref{max_principal}, there is a sequence $\{x_k\}_{k=1}^\infty\subset\bR^d$ with
		\begin{equation}\label{x_k}
			\lvert x_k\rvert = R + k,\quad\text{and}\quad |u(x_k)|\geq |u(x_0)|.
		\end{equation}
		
		Let us set
		\begin{equation*}
			\psi_k(r)=\int_{\partial B_1(0)}u^2(r x + x_k)dS=r^{-(d-1)}\int_{\partial B_r(x_k)} u^2 dS\geq 0.
		\end{equation*}
		Then it holds that
		\begin{align*}
			\psi_k'(r)&=\int_{\partial B_1(0)}\frac{\partial}{\partial r}u^2(r x + x_k)dS=\int_{\partial B_1(0)}\left\langle \nabla u^2(r x + x_k),x\right\rangle dS\\
			&= r^{-(d-1)} \int_{\partial B_r(0)} \left\langle\nabla u^2(x+x_k),\frac{x}{r}\right\rangle dS =r^{-(d-1)} \int_{B_r(x_k)}\Delta (u^2) dx \\
			&=2 r^{-(d-1)} \int_{B_r(x_k)} (u \Delta u + |\nabla u|^2) dx\geq 2 r^{-(d-1)} \int_{B_r(x_k)} u \Delta u dx\\
			&= 2 r^{-(d-1)} \int_{B_r(x_k)} V u^2 dx.
		\end{align*}
		Note that $V(x)\geq\frac{\alpha}{2}$ holds on $B_k(x_k)$. We have for any $r\leq k$ that 
		\begin{align*}
			\psi_k'(r)&\geq 2 r^{-(d-1)} \int_{B_r(x_k)} V u^2 dx \geq \alpha r^{-(d-1)} \int_{B_r(x_k)} u^2 dx\\
			&= \alpha r^{-(d-1)} \int_0^r \int_{\partial B_t(x_k)} u^2 dS dt= \alpha \int_0^r \left(\frac{t}{r}\right)^{d-1}\psi_k(t)dt\geq 0,
		\end{align*}
		which implies that $\psi_k$ is monotonically increasing on $[0,k]$. For any $n\in\{1,2,\dots, k-1\}$ and any $r\in[n,n+1]$, we have that
		\begin{equation*}
			\begin{split}
				\psi_k'(r)&\geq \alpha \int_0^r \left(\frac{t}{r}\right)^{d-1}\psi_k(t)dt\geq \alpha \int_n^r \left(\frac{t}{r}\right)^{d-1}\psi_k(t)dt\\
				&\geq \alpha (r-n)\cdot\left(\frac{n}{n+1}\right)^{d-1}\psi_k(n),
			\end{split}
		\end{equation*}
		and hence that
		\begin{align*}
			\psi_k(n+1) & = \psi_k(n) + \int_n^{n+1} \psi_k'(r)dr\\
			&\geq \left(1+ \alpha \int_n^{n+1}(r-n)dr \cdot\left(\frac{n}{n+1}\right)^{d-1}\right)\psi_k(n)\\
			&\geq \left(1+ \frac{\alpha}{2^d}\right)\cdot\psi_k(n).
		\end{align*}
		Thus, it holds that
		\begin{equation}\label{psi_k}
			\begin{split}
				\psi_k(k)&\geq \left(1+ \frac{\alpha}{2^d}\right)^{k-1}\cdot\psi_k(1)\geq \left(1+ \frac{\alpha}{2^d}\right)^{k-1}\cdot\psi_k(0)\\
				&=\left(1+ \frac{\alpha}{2^d}\right)^{k-1}\cdot u(x_k)^2\geq \left(1+ \frac{\alpha}{2^d}\right)^{k-1}\cdot u(x_0)^2,\quad\forall~k\in\bNp,
			\end{split}
		\end{equation}
		where we used the monotone property of $\psi_k$ on $[0,k]$ and \eqref{x_k}.
		
		Note that $u\in L^\infty(\bR^d)$. So $\{\psi(k)\}_{k=1}^\infty$ must be bounded, which contradicts \eqref{psi_k} as $u(x_0)\neq 0$. We therefore can conclude that $u=0$, which proves the injectivity of $I+\calT_{\alpha,W}$.
	\end{proof}
	
	\begin{remark}
		We remark that the standard proof of the uniqueness of $H^1$-solutions to elliptic PDEs in  dose not apply to the Barron solutions. In fact, the uniqueness in $H^1(\bR^d)$ of solutions of the equation $-\Delta u + Vu =0$  follows from a standard energy estimate. Noticing that $-\Delta u + Vu =0 \in H^{-1}(\bR^d)$ that is the dual space of $H^1(\bR)$, one has that $0 = \int_{\bR^d} (-\Delta u + Vu) u = \int_{\bR^d} |\nabla u|^2 + V u^2$, which implies $u=0$.  However, such energy estimate in general does not apply to Barron functions in $\calB^s(\bR^d)$ since in general $\calB^s(\bR^d) \nsubseteq H^1(\bR^d)$. To give a concrete example, let us consider the function $u$ whose  Fourier transform is defined by  $$\hat{u}(\xi_1,\dots,\xi_d) = \begin{cases}|\xi_1|^{-\frac{1}{2}}, & \text{if }\xi\in([-1,1]\backslash\{0\})\times [-1,1]^{d-1},\\
			0, &\text{otherwise}.\end{cases}$$ Then $\int_{\bR^d} |\hat{u}(\xi)|\cdot (1+|\xi|^2)^{s/2}d\xi<\infty$ while $\int_{\bR^d} |\hat{u}(\xi)|^2 \cdot(1+|\xi|^2)d\xi = \infty$, i.e., $u\in \calB^s(\bR^d)\backslash H^1(\bR^d)$. 
	\end{remark}
	
	Proposition~\ref{bdd_inverse_prop} is then a direct corollary.
	
	\begin{proof}[Proof of Proposition~\ref{bdd_inverse_prop}]
		The result follows directly from Proposition~\ref{compact_prop}, Proposition~\ref{injective_prop}, and the Fredholm alternative.
	\end{proof}
	
	We can finally prove Theorem~\ref{main_thm}.
	
	\begin{proof}[Proof of Theorem~\ref{main_thm}]
		Let $u^*$ be the unique solution to \eqref{schrodinger_equ} or \eqref{integral_equ}. Notice that by Proposition~\ref{bdd_inverse_prop} and Lemma~\ref{prelim_lemma1},
		\begin{equation*}
			u^* = (I+\calT_{\alpha,W})^{-1}\left((\alpha-\Delta)^{-1}f\right)\in\calB^s(\bR^d),
		\end{equation*}
		with
		\begin{equation}\label{u_star1}
			\begin{split}
				\norm{u^*}_{\calB^s(\bR^d)}&=\norm{(I+\calT_{\alpha,W})^{-1}\left((\alpha-\Delta)^{-1}f\right)}_{\calB^s(\bR^d)}\\
				&\leq \frac{1}{\alpha}\norm{(I+\calT_{\alpha,W})^{-1}}_{\calB^s(\bR^d)\rightarrow\calB^s(\bR^d)}\norm{f}_{\calB^s(\bR^d)}.
			\end{split}
		\end{equation}
		It follows from Lemma~\ref{prelim_lemma1}, Lemma~\ref{prelim_lemma2}, and \eqref{schrodinger_equ} that
		\begin{equation}\label{u_star2}
			\begin{split}
				\norm{u^*}_{\calB^{s+2}(\bR^d)}&\leq \frac{1}{\min\{\alpha, 1\}}\norm{(\alpha-\Delta)u^*}_{\calB^s(\bR^d)}\\
				&=\frac{1}{\min\{\alpha, 1\}}\norm{W u^*-f}_{\calB^s(\bR^d)}\\
				&\leq \frac{1}{\min\{\alpha,1\}}(\norm{W u^*}_{\calB^s(\bR^d)}+\norm{f}_{\calB^s(\bR^d)})\\
				&\leq \frac{1}{\min\{\alpha,1\}}\left(2^{\frac{s}{2}}\norm{W}_{\calB^s(\bR^d)}\norm{u^*}_{\calB^s(\bR^d)}+\norm{f}_{\calB^s(\bR^d)}\right).
			\end{split}
		\end{equation}
		Combining \eqref{u_star1} and \eqref{u_star2}, we obtain that
		\begin{equation*}
			\norm{u^*}_{\calB^{s+2}(\bR^d)}\leq \frac{1}{\min\{\alpha,1\}}\left(\frac{2^{\frac{s}{2}}\norm{W}_{\calB^s(\bR^d)}}{\alpha}\norm{(I+\calT_{\alpha,W})^{-1}}_{\calB^s(\bR^d)\rightarrow\calB^s(\bR^d)}+1\right)\norm{f}_{\calB^s}.
		\end{equation*}
		Hence, \eqref{regularity_esti} holds with
		\begin{equation*}
			C=\frac{1}{\min\{\alpha,1\}}\left(\frac{2^{\frac{s}{2}}\norm{W}_{\calB^s(\bR^d)}}{\alpha}\norm{(I+\calT_{\alpha,W})^{-1}}_{\calB^s(\bR^d)\rightarrow\calB^s(\bR^d)}+1\right),
		\end{equation*}
		which completes the proof.
	\end{proof}
	
	Theorem~\ref{approx_thm} then follows directly from Theorem~\ref{main_thm} and some techniques for establishing approximation without CoD in previous literature.
	
	\begin{proof}[Proof of Theorem~\ref{approx_thm}]
		This proof uses techniques from \cites{Barron93, E19}, and is similar to \cite{chen2021representation}*{Theorem 2.5}. Note that $s\geq 0$. According to Theorem~\ref{main_thm} and Lemma~\ref{prelim_lemma0}, it holds that
		\begin{equation*}
			\norm{u^*}_{\calB^0(\bR^d)}\leq \norm{u^*}_{\calB^2(\bR^d)}\leq \norm{u^*}_{\calB^{s+2}(\bR^d)}\leq C \norm{f}_{\calB^s(\bR^d)}.
		\end{equation*}
		Denote $\hat{u^*}(\xi)=|\hat{u^*}(\xi)| e^{i\theta(\xi)}$ and let $\mu$ be a probability distribution on $\bR^d$ with density being $|\hat{u^*}(\xi)|/\norm{\hat{u^*}}_{L^1(\bR^d)}=|\hat{u^*}(\xi)|/\norm{u^*}_{\calB^0(\bR^d)}$. Then the real-valued function $u^*$ can be represented as 
		\begin{equation*}
			\begin{split}
				u^*(x)&=\int_{\bR^d}\hat{u^*}(\xi) e^{i \xi^T x}d\xi=\int_{\bR^d}|\hat{u^*}(\xi)| e^{i (\xi^T x + \theta(\xi))}d\xi \\
				& =\int_{\bR^d}|\hat{u^*}(\xi)| \cos(\xi^T x + \theta(\xi)) d\xi
				= \norm{u^*}_{\calB^0(\bR^d)} \bE_{\xi\sim\mu} \left[\cos(\xi^T x + \theta(\xi))\right].
			\end{split}
		\end{equation*}
		Note that $\hat{u^*}\in\calB^2(\bR^d)$, which implies that $\mu$ has finite first-order and second-order moment. Therefore,
		\begin{equation*}
			\frac{\partial}{\partial x_k} u^*(x) = - \norm{u^*}_{\calB^0(\bR^d)} \bE_{\xi\sim\mu}[\langle \xi, e_k\rangle \sin(\xi^T x + \theta(\xi))].
		\end{equation*}
		Let $\xi_1,\xi_2,\dots,\xi_n$ be i.i.d. samples from $\mu$, and let
		\begin{equation*}
			u_n(x)=\frac{1}{n}\sum_{j=1}^n a_j\cos(w_j^\top x + b_j),
		\end{equation*}
		where $a_j=\norm{u^*}_{\calB^0(\bR^d)}$, $w_j=\xi_j$, and $b_j=\theta(\xi_j)$. Then it holds that
		\begin{align*}
			& \bE_{\mu^{\otimes n}}\norm{u^*-u_n}_{H^1(\Omega)}^2\\
			= & \bE_{\mu^{\otimes n}}\left[\int_\Omega |u^*(x)-u_n(x)|^2 dx + \sum_{k=1}^d \int_\Omega \left|\frac{\partial}{\partial x_k} u^*(x)-\frac{\partial}{\partial x_k} u_n(x)\right|^2 dx\right]\\
			= & \norm{u^*}_{\calB^0(\bR^d)}^2 \int_\Omega \text{Var}_{\mu^{\otimes n}} \left[\frac{1}{n}\sum_{j=1}^n \cos(\xi_j^\top x + \theta(\xi_j))\right]dx \\
			& + \norm{u^*}_{\calB^0(\bR^d)}^2 \sum_{k=1}^d \int_\Omega \text{Var}_{\mu^{\otimes n}} \left[\frac{1}{n}\sum_{j=1}^n \langle \xi_j,e_k\rangle \sin(\xi_j^\top x + \theta(\xi_j))\right]dx \\
			= & \frac{\norm{u^*}_{\calB^0(\bR^d)}^2}{n} \int_\Omega \text{Var}_{\xi\sim\mu} \left[\cos(\xi^\top x + \theta(\xi))\right]dx \\
			& + \frac{\norm{u^*}_{\calB^0(\bR^d)}^2}{n} \sum_{k=1}^d \int_\Omega \text{Var}_{\xi\sim\mu} \left[\langle \xi,e_k\rangle \sin(\xi^\top x + \theta(\xi))\right]dx \\ 
			\leq & \frac{\norm{u^*}_{\calB^0(\bR^d)}^2}{n} \int_\Omega \bE_{\xi\sim\mu}\left[1+\sum_{k=1}^d \langle \xi,e_k\rangle^d\right]dx\\
			= &  \frac{\norm{u^*}_{\calB^0(\bR^d)}}{n} \int_\Omega \int_{\bR^d}|\hat{u^*}(\xi)|\cdot(1+\lvert\xi\rvert^2)d\xi dx \\
			\leq &\frac{m(\Omega) \norm{u^*}_{\calB^0(\bR^d)} \norm{u^*}_{\calB^2(\bR^d)}}{n} \\
			\leq & \frac{m(\Omega)\cdot C^2 \norm{f}_{\calB^s(\bR^d)}^2}{n},
		\end{align*}
		which implies \eqref{approx_equ}.
	\end{proof}

	\bibliographystyle{amsxport}
	\bibliography{references}

\begin{bibdiv}
\begin{biblist}

\bib{Bach17}{article}{
      author={Bach, Francis},
       title={Breaking the curse of dimensionality with convex neural
  networks},
        date={2017},
     journal={Journal of Machine Learning Research},
      volume={18},
      number={1},
       pages={629\ndash 681},
}

\bib{Barron93}{article}{
      author={Barron, Andrew~R},
       title={Universal approximation bounds for superpositions of a sigmoidal
  function},
        date={1993},
     journal={IEEE Transactions on Information Theory},
      volume={39},
      number={3},
       pages={930\ndash 945},
}

\bib{chen2021representation}{inproceedings}{
      author={Chen, Ziang},
      author={Lu, Jianfeng},
      author={Lu, Yulong},
       title={On the representation of solutions to elliptic {PDE}s in {B}arron
  spaces},
        date={2021},
   booktitle={Advances in neural information processing systems},
      volume={34},
}

\bib{weinan2017deep}{article}{
      author={E, Weinan},
      author={Han, Jiequn},
      author={Jentzen, Arnulf},
       title={Deep learning-based numerical methods for high-dimensional
  parabolic partial differential equations and backward stochastic differential
  equations},
        date={2017},
     journal={Communications in Mathematics and Statistics},
      volume={5},
      number={4},
       pages={349\ndash 380},
}

\bib{weinan2021algorithms}{article}{
      author={E, Weinan},
      author={Han, Jiequn},
      author={Jentzen, Arnulf},
       title={Algorithms for solving high dimensional pdes: From nonlinear
  monte carlo to machine learning},
        date={2021},
     journal={Nonlinearity},
      volume={35},
      number={1},
       pages={278},
}

\bib{E19}{article}{
      author={E, Weinan},
      author={Ma, Chao},
      author={Wu, Lei},
       title={The {B}arron space and the flow-induced function spaces for
  neural network models},
        date={2021},
     journal={Constructive Approximation},
       pages={1\ndash 38},
}

\bib{E20}{inproceedings}{
      author={E, Weinan},
      author={Wojtowytsch, Stephan},
       title={Some observations on high-dimensional partial differential
  equations with {B}arron data},
organization={PMLR},
        date={2022},
   booktitle={Mathematical and scientific machine learning},
       pages={253\ndash 269},
}

\bib{weinan2018deep}{article}{
      author={E, Weinan},
      author={Yu, Bing},
       title={The deep {R}itz method: a deep learning-based numerical algorithm
  for solving variational problems},
        date={2018},
     journal={Communications in Mathematics and Statistics},
      volume={6},
      number={1},
       pages={1\ndash 12},
}

\bib{folland1999real}{book}{
      author={Folland, Gerald~B},
       title={Real analysis: modern techniques and their applications},
   publisher={John Wiley \& Sons},
        date={1999},
      volume={40},
}

\bib{han2018solving}{article}{
      author={Han, Jiequn},
      author={Jentzen, Arnulf},
      author={E, Weinan},
       title={Solving high-dimensional partial differential equations using
  deep learning},
        date={2018},
     journal={Proceedings of the National Academy of Sciences},
      volume={115},
      number={34},
       pages={8505\ndash 8510},
}

\bib{elliptic_PDE_fanghua}{book}{
      author={Han, Qing},
      author={Lin, Fang-Hua},
       title={Elliptic partial differential equations},
     edition={2nd ed.},
      series={Courant lecture notes},
   publisher={Courant Institute of Mathematical Sciences, Robotics Lab, New
  York University},
        date={2011},
        ISBN={0821853139},
}

\bib{hanche2010kolmogorov}{article}{
      author={Hanche-Olsen, Harald},
      author={Holden, Helge},
       title={The kolmogorov--riesz compactness theorem},
        date={2010},
     journal={Expositiones Mathematicae},
      volume={28},
      number={4},
       pages={385\ndash 394},
}

\bib{lu2021priori2}{article}{
      author={Lu, Jianfeng},
      author={Lu, Yulong},
       title={A priori generalization error analysis of two-layer neural
  networks for solving high dimensional {S}chr{\"o}dinger eigenvalue problems},
        date={2022},
     journal={Communications of the American Mathematical Society},
      volume={2},
      number={01},
       pages={1\ndash 21},
}

\bib{lu2021priori}{inproceedings}{
      author={Lu, Yulong},
      author={Lu, Jianfeng},
      author={Wang, Min},
       title={A priori generalization analysis of the deep {R}itz method for
  solving high dimensional elliptic partial differential equations},
organization={PMLR},
        date={2021},
   booktitle={Conference on learning theory},
       pages={3196\ndash 3241},
}

\bib{Marwah21}{inproceedings}{
      author={Marwah, Tanya},
      author={Lipton, Zachary~C},
      author={Risteski, Andrej},
       title={Parametric complexity bounds for approximating {PDE}s with neural
  networks},
        date={2021},
   booktitle={Advances in neural information processing systems},
      volume={34},
}

\bib{parhi2021banach}{article}{
      author={Parhi, Rahul},
      author={Nowak, Robert~D},
       title={Banach space representer theorems for neural networks and {R}idge
  splines},
        date={2021},
     journal={Journal of Machine Learning Research},
      volume={22},
      number={43},
       pages={1\ndash 40},
}

\bib{parhi2021near}{article}{
      author={Parhi, Rahul},
      author={Nowak, Robert~D},
       title={Near-minimax optimal estimation with shallow {R}elu neural
  networks},
        date={2021},
     journal={arXiv preprint arXiv:2109.08844},
}

\bib{siegel2021optimal}{article}{
      author={Siegel, Jonathan~W},
      author={Xu, Jinchao},
       title={Sharp bounds on the approximation rates, metric entropy, and
  $n$-widths of shallow neural networks},
        date={2021},
     journal={arXiv preprint arXiv:2101.12365},
}

\bib{siegel2020high}{article}{
      author={Siegel, Jonathan~W},
      author={Xu, Jinchao},
       title={High-order approximation rates for shallow neural networks with
  cosine and ${ReLU}^k$ activation functions},
        date={2022},
     journal={Applied and Computational Harmonic Analysis},
      volume={58},
       pages={1\ndash 26},
}

\end{biblist}
\end{bibdiv}

\end{document}